\documentclass[11pt,reqno]{amsart}
\usepackage[utf8]{inputenc}

\usepackage[margin=1in]{geometry}
\usepackage{amsmath}
\usepackage{amsfonts}
\usepackage{amssymb}
\usepackage{amsthm}
\usepackage{enumitem}
\usepackage{mathrsfs}
\usepackage{mathtools}
\usepackage{tikz-cd}
\usepackage{enumitem}

\usepackage[colorlinks=true,hyperindex, linkcolor=magenta, pagebackref=false, citecolor=cyan,pdfpagelabels]{hyperref}

\newtheorem{theorem}{Theorem}[section]

\newtheorem{lemma}[theorem]{Lemma}

\theoremstyle{definition}

\newtheorem{remark}[theorem]{Remark}

\newcommand{\bZ}{{\mathbf Z}}

\newcommand{\F}{{\mathbf F}}
\newcommand{\Q}{{\mathbf Q}}
\newcommand{\Z}{{\mathbf Z}}

\newcommand{\cE}{\mathcal{E}}

\DeclareMathOperator{\ad}{ad}

\DeclareMathOperator{\BC}{BC}

\DeclareMathOperator{\Fr}{Fr}

\DeclareMathOperator{\Gal}{Gal}

\DeclareMathOperator{\Irr}{Irr}

\DeclareMathOperator{\Rep}{Rep}
\DeclareMathOperator{\Res}{Res}
\DeclareMathOperator{\rk}{rk}

\DeclareMathOperator{\Transp}{Transp}

\newcommand{\wtil}[1]{\wtil}

\newcommand{\ol}[1]{\overline{#1}}
\newcommand{\co}[1]{\colon}

\title{Base change and the Mackey formula in Deligne--Lusztig theory}
\author{Sean Cotner}

\begin{document}

\begin{abstract}
    Using ideas of Digne and recent work of Gaitsgory--Rozenblyum--Varshavsky, we show that Lusztig induction commutes with (a minor variant of) the Glauberman correspondence for large prime degree field extensions. This allows one to reduce the Mackey formula in Deligne--Lusztig theory to the case of ``large'' base fields, in which case the result is due to Bonnaf\'e.
\end{abstract}

\maketitle

\section{Introduction}

Let $G$ be a connected reductive group over a finite field $\F_q$. For a twisted Levi $\F_q$-subgroup $L \subset G$ whose base extension $L_{\ol\F_q}$ is a Levi factor of a parabolic $\ol\F_q$-subgroup $P \subset G_{\ol\F_q}$, Lusztig defined the \emph{Lusztig induction} $R_{L, P}^G\colon \Rep(L(\F_q)) \to \Rep(G(\F_q))$ and its adjoint \emph{Lusztig restriction} $^*R^G_{L, P}$, generalizing Harish--Chandra induction and restriction. This generalization suggests the \emph{Mackey formula}, which states that if $(L, P)$ and $(M, Q)$ are two pairs as above then
\begin{equation}\label{eqn:mackey-formula}
    ^*R^G_{L, P} \circ R_{M, Q}^G = \sum_{w \in L(\F_q)\backslash S(L,M)(\F_q)/M(\F_q)} R_{L \cap {}^wM, L \cap {}^wQ}^L \circ {}^*R_{L \cap {}^wM, P\cap {}^wM}^{^wM} \circ \ad w
\end{equation}
in the Grothendieck group $K_0(\Rep(L(\F_q)))$, where $S(L,M)$ is the subscheme of $G$ parameterizing those elements $g$ such that $L$ and $^gM$ share a common maximal torus.

The Mackey formula \eqref{eqn:mackey-formula} is a natural statement which shows up in many calculations involving Lusztig induction, since (by adjunction) it governs the pairings of two Lusztig inductions. As such, it is a common hypothesis in results in the literature; in particular, it is used in modular representation theory through its role in generalized Harish--Chandra theory. See for instance \cite{Eng13}, \cite[\S 4.7]{GM20}, \cite{Hol22}, and \cite{Ros25} for theorems assuming the Mackey formula. The importance of the Mackey formula is also exposed in detail in \cite{DM20}.

The validity of the Mackey formula \eqref{eqn:mackey-formula} is a long-standing open problem in Deligne--Lusztig theory, but it is now known in most cases:
\begin{enumerate}
    \item $L$ or $M$ is a maximal torus \cite[Theorem 7]{DL83} (building on \cite[Theorem 6.8]{DL76}),
    \item $P$ and $Q$ are defined over $\F_q$ \cite[Lemma 2.5]{LS79} (attributed to Deligne),
    \item $q$ is ``large'' \cite[Th\'eor\`eme 5.1.1]{Bon98} (also due to Deligne in unpublished work from 1975; Lusztig has informed us that an ineffective version also follows from \cite{Lus85} and \cite{Lus90}),
    \item $q > 2$ or $G$ does not contain a simple component of types $^2\mathrm{E}_6$, $\mathrm{E}_7$, or $\mathrm{E}_8$ \cite[Theorem 3.9]{BM11} (see also \cite[Theorem 5.2.1]{Bon00} for the result in type A),
    \item $q = 2$ and the center $Z(G)$ is connected and $G$ has no simple factor of type $\mathrm{E}_8$ \cite[Theorem 1.6]{Tay18}.
\end{enumerate}
The proofs of (1), (2), and (3) proceed along general principles, while the proof of (4) is ``as ugly as possible'' in the words of its authors, relying on extensive computer calculations and strange properties of the set of semisimple elements of $G(\F_q)$. The proof of (5) is based on that of (4). The main aim of this paper is to give a conceptual proof of the following theorem.

\begin{theorem}\label{theorem:main}
    The Mackey formula \eqref{eqn:mackey-formula} is true.
\end{theorem}

One reason for interest in the Mackey formula \eqref{eqn:mackey-formula} is that it implies that $R_{L,P}^G$ is independent of $P$. This was recently proven (without the Mackey formula) in work of Gaitsgory--Rozenblyum--Varshavsky \cite[Theorem 8.7.2]{GRV26}, and we will use this theorem crucially in what follows. In particular, from now on we will write (for instance) $R_L^G$ in place of $R_{L,P}^G$. Our proof also uses cases (1) and (3) above of the Mackey formula, but it does not use cases (2), (4), or (5).

We now give a few details about the proof. If $\ell$ is a prime number not dividing the order of $G(\F_q)$, then it is easy to see (Lemma~\ref{lemma:size-of-finite-group-of-lie-type}) that $\ell$ also does not divide the order of $G(\F_{q^\ell})$. The Glauberman correspondence \cite[Theorem 3]{Gla68} shows that there exists a natural ``base change'' isometry
\[
\BC_\ell\colon K_0(\Rep(G(\F_q))) \to K_0(\Rep(G(\F_{q^\ell})))^{\Gal(\F_{q^\ell}/\F_q)}.
\]
The first observation is that, by a result of Digne on Shintani descent \cite[Corollaire 3.5]{Dig99}, if $\rho$ is an irreducible representation of $L(\F_q)$ then the class function $R_{L_{\F_{q^\ell}}}^{G_{\F_{q^\ell}}}(\BC_\ell(\rho))$ on $G(\F_{q^\ell})$ admits an extension to a class function $\widetilde\psi$ on $G(\F_{q^\ell}) \rtimes \Gal(\F_{q^\ell}/\F_q)$ with the property that, if $\sigma$ is a generator of $\Gal(\F_{q^\ell}/\F_q)$, then
\[
\widetilde\psi(g \rtimes \sigma) = R_L^G(\rho)(g)
\]
for all $g \in G(\F_q)$. Actually, Digne's result requires the characteristic of $\F_q$ to be good for $G$, a condition which does not hold in the remaining open cases of the Mackey formula. However, the assumption on the characteristic only appears in Digne's proof because of features of Shintani descent which do not appear for the Glauberman correspondence\footnote{The Glauberman correspondence is not quite a special case of Shintani descent; in fact, it is a ``twisted'' version of a special case of Shintani descent. (See \cite[Corollaire 2.5]{Dig86}.) This twisting is really the key point in generalizing Digne's result, but we only know how to perform it for large $\ell$.}; see Lemma~\ref{lemma:psi-tilde} for an extension.

The above discussion suggests the following theorem.

\begin{theorem}\label{thm:base-change-commutes-with-lusztig-induction}
    There is a constant $D_G$, depending only on the root datum of $G$, such that for all twisted Levi $\F_q$-subgroups $L \subset G$ and all prime numbers $\ell > D_G$, we have
    \begin{equation}\label{eqn:bc-commutes-with-dl-induction}
    \BC_\ell \circ R_L^G = R_{L_{\F_{q^\ell}}}^{G_{\F_{q^\ell}}} \circ \BC_\ell\colon K_0(\Rep(L(\F_q))) \to K_0(\Rep(G(\F_{q^\ell})))
    \end{equation}
    and
    \begin{equation}\label{eqn:bc-commutes-with-dl-restriction}
    \BC_\ell \circ {}^*R^G_L = {}^*R^{G_{\F_{q^\ell}}}_{L_{\F_{q^\ell}}} \circ \BC_\ell\colon K_0(\Rep(G(\F_q))) \to K_0(\Rep(L(\F_{q^\ell}))).
    \end{equation}
\end{theorem}

With Theorem~\ref{thm:base-change-commutes-with-lusztig-induction} in hand, the proof of Theorem~\ref{theorem:main} is straightforward: since $\BC_\ell$ is injective, it suffices to prove \eqref{eqn:mackey-formula} after composition with $\BC_\ell$, and we can thereby pass to the case that $q$ is ``large'', a case already handled by Bonnaf\'e \cite[Th\'eor\`eme 5.1.1]{Bon98}.

To prove Theorem~\ref{thm:base-change-commutes-with-lusztig-induction}, the difficulty of simply applying Digne's above result is that the class functions $\BC_\ell(R_L^G(\rho))$ and $R_{L_{\F_{q^\ell}}}^{G_{\F_{q^\ell}}}(\BC_\ell(\rho))$ are extracted in rather different ways from $\widetilde\psi$, and a priori there could be considerable (and different) ``cancellations'' in both extractions. To show that this does not occur, we show that for large $\ell$ the decomposition of $\widetilde\psi$ into a $\Z$-linear combination of irreducible representations simply does not have large enough coefficients to permit such cancellation. The main technical ingredients involve bounding the size of Lusztig series (Lemma~\ref{lemma:lusztig-series-size-bound}) and the self-pairings of Lusztig induction for large $q$ (Lemma~\ref{lemma:mackey-formula-bound}).

A basic philosophy of this paper, which is useful for some other problems, is that for large $\ell$, base change from $\F_q$ to $\F_{q^\ell}$ ``creates more room'' in $G(\F_q)$ but otherwise does not affect the essential group-theoretic properties of $G$. For instance, such base change does not affect the set of rational conjugacy classes of maximal tori or the Weyl groups of these maximal tori (Lemma~\ref{lemma:tori-descend}) or individual geometric Lusztig series (Lemma~\ref{lemma:glauberman-lusztig-series-bijection}).

\subsection{Notation}

Throughout this paper, $G$ is a connected reductive group over $\F_q$, and $\ell$ is a prime number. The notation $\Fr_q$ is used to denote the Frobenius automorphism of $\ol\F_q$, as well as (in one instance) the Frobenius endomorphism of an algebraic $\F_q$-group.

\subsection{Acknowledgements}

This paper grew out of joint work with Tony Feng on modular functoriality in the local Langlands program. I thank Tony for a fruitful collaboration. I also thank Jay Taylor for helpful comments. The author did not employ AI tools in the preparation of this work.
This material is based upon work supported by the National Science Foundation under Award No.\ 2402231 and the European Research Council (ERC) under the European Union’s Horizon 2020 research and innovation programme (grant agreement no.\ 950326).

\section{Preliminaries on base change and Lusztig induction}

The \emph{Glauberman correspondence} \cite[Theorem 3]{Gla68} states (in a special case) that if $\Gamma$ is a finite group equipped with an automorphism $\alpha$ of order prime to $|\Gamma|$, then there is a natural bijection $\Irr(\Gamma^\alpha) \cong \Irr(\Gamma)^\alpha$ of sets of irreducible characters. This bijection is characterized by the property that if $\chi_0 \in \Irr(\Gamma^\alpha)$ corresponds to $\chi \in \Irr(\Gamma)^\alpha$, then there is some $\epsilon \in \{\pm 1\}$ such that
\[
\chi_0(\gamma) = \epsilon\widetilde\chi(\gamma \rtimes \alpha)
\]
for all $\gamma \in \Gamma^\alpha$, where $\widetilde\chi$ is the unique extension of $\chi$ to a character of $\Gamma \rtimes \langle\alpha\rangle$ such that $\widetilde\chi(1 \rtimes \alpha) \in \bZ$.

The following (presumably well-known) lemma shows that the Glauberman correspondence can be applied easily in our setting.

\begin{lemma}\label{lemma:size-of-finite-group-of-lie-type}
    Let $q$ be a prime power, let $X$ be a finite type $\F_q$-scheme, and let $\ell$ be a prime number not dividing $|X(\F_q)|$. Then $\ell$ does not divide $|X(\F_{q^\ell})|$.
\end{lemma}

\begin{proof}
    By the Weil conjectures, there exist algebraic integers $\alpha_1, \dots, \alpha_M, \beta_1, \dots, \beta_N$ such that for all $n \geq 1$ we have $|X(\F_{q^n})| = \sum_{i=1}^M \alpha_i^n - \sum_{j=1}^N \beta_j^n$. Since $\ell$ does not divide $|X(\F_q)|$ by hypothesis, the sum $\sum_{i=1}^M \alpha_i - \sum_{j=1}^N \beta_j$ is an $\ell$-adic unit. Since 
    \[
    \sum_{i=1}^M \alpha_i^\ell - \sum_{j=1}^N \beta_j^\ell \equiv \left( \sum_{i=1}^M \alpha_i - \sum_{j=1}^N \beta_j\right)^\ell \pmod{\ell},
    \]
    it follows that $\sum_{i=1}^M \alpha_i^\ell - \sum_{j=1}^N \beta_j^\ell$ is also an $\ell$-adic unit, i.e., $\ell$ does not divide $|X(\F_{q^\ell})|$.\footnote{In place of the Weil conjectures, one could also use the known explicit formulas for the order of $G(\F_{q^n})$.}
\end{proof}

Now suppose $\ell$ is a prime number not dividing the order of $G(\F_q)$, and let $\sigma$ denote the order $\ell$ automorphism of $G(\F_{q^\ell})$ induced by a generator of $\Gal(\F_{q^\ell}/\F_q)$. By Lemma~\ref{lemma:size-of-finite-group-of-lie-type}, the group $G(\F_{q^\ell})$ is also of order prime to $\ell$. We define a homomorphism
\[
    \BC_\ell'\colon K_0(\Rep(G(\F_q))) \to K_0(\Rep(G(\F_{q^\ell})\rtimes\langle\sigma\rangle))
\]
by ``uncorrecting the sign'' in the Glauberman correspondence; in other words, if $\chi$ is an irreducible character of $G(\F_q)$ then we define $\BC_\ell'(\chi)$ to be the unique class function on $G(\F_{q^\ell}) \rtimes\langle\sigma\rangle$ such that $\pm\BC_\ell'(\chi)$ is an irreducible character and
\[
\BC_\ell'(\chi)(g\rtimes\sigma) = \chi(g)
\]
for all $g \in G(\F_q)$. We let $\BC_\ell\colon K_0(\Rep(G(\F_q))) \to K_0(\Rep(G(\F_{q^\ell})))$ be the composition of $\BC_\ell'$ with the forgetful map. By the usual Glauberman correspondence, the map $\BC_\ell$ is an (injective) isometry with respect to the usual inner products on both sides.

The following lemma (partially) generalizes \cite[Corollaire 3.5]{Dig99} to arbitrary $q$. Let $L \subset G$ be a twisted Levi $\F_q$-subgroup, let $G_\ell = \Res_{\F_{q^\ell}/\F_q}(G_{\F_{q^\ell}})$, and similarly for $L_\ell$. Let $P \subset G_{\ol\F_q}$ be a parabolic $\ol\F_q$-subgroup with Levi factor $L_{\ol\F_q}$ and unipotent radical $U$, and let $P_1$ denote the parabolic $\ol\F_q$-subgroup of $(G_\ell)_{\ol\F_q} \cong (G_{\ol\F_q})^\ell$ with Levi factor $(L_\ell)_{\ol\F_q}$ and unipotent radical $(U_{\ol\F_q})^\ell$. We will freely use the generalization of Deligne--Lusztig theory to disconnected groups which is developed in \cite{DM94}.

\begin{lemma}\label{lemma:dig99-3.5}
    Let $\rho$ be an irreducible representation of $L(\F_q)$. Then
    \[
    R_{L_\ell\rtimes\langle\sigma\rangle, P_1 \rtimes\langle\sigma\rangle}^{G_\ell\rtimes \langle\sigma\rangle}(\BC_\ell'(\rho))(g \rtimes \sigma) = R_L^G(\rho)(g)
    \]
    for all $g \in G(\F_q)$.
\end{lemma}

\begin{proof}
    Let $\rho$ be an irreducible representation of $L(\F_q)$, and let $Q_L$ denote the Green function defined in the statement of \cite[Proposition 2.6]{DM94}. By \cite[Theorem 8.7.2]{GRV26}, the function $Q_L$ is independent of the choice of $P$, justifying the notation. The character formula of \cite[Proposition 2.6]{DM94} shows that if $g \in G(\F_q)$ has Jordan decomposition $g = su$, then
    \begin{align*}
    &R_{L_\ell \rtimes\langle\sigma\rangle, P_1 \rtimes \langle\sigma\rangle}^{G_\ell\rtimes\langle\sigma\rangle}(\BC_\ell'(\rho))(g \rtimes\sigma)\\
    &= \frac{1}{\ell \cdot|Z_G(s)^\circ(\F_q)|} \sum_{\substack{h = h_0 \rtimes \sigma^i \in G(\F_{q^\ell}) \rtimes \langle\sigma\rangle \\ s \rtimes \sigma \in {}^h(L(\F_{q^\ell})\rtimes\langle\sigma\rangle)}} \sum_{v \in Z_{^hL}(s)(\F_q)_{\mathrm{u}}} Q_{(({}^hL_\ell)^\sigma)^\circ}^{((G_\ell)^{s \rtimes \sigma})^\circ}(u, v^{-1}) \cdot \BC'_\ell(\rho)(h_0^{-1}sv\sigma(h_0) \rtimes \sigma), \addtocounter{equation}{1}\tag{\theequation} \label{eqn:char-formula-digne-michel}
    \end{align*}
    where the subscript $\mathrm{u}$ refers to the set of unipotent elements. Note that our choice of notation differs from that used in \cite[Proposition 2.6]{DM94}, where $su$ is used for the Jordan decomposition of $g \rtimes \sigma$; the justification for the discrepancy is that, by our hypothesis and Lemma~\ref{lemma:size-of-finite-group-of-lie-type}, the group $G(\F_{q^\ell})$ is of order not divisible by $\ell$, so $1 \rtimes \sigma$ and $g$ are both powers of $g \rtimes \sigma$ since $g$ commutes with $\sigma$. In particular, \eqref{eqn:char-formula-digne-michel} is equivalent to \cite[Proposition 2.6]{DM94}.

    We now aim to simplify the rather fearsome equation \eqref{eqn:char-formula-digne-michel}. Fix $g = su \in G(\F_q)$ as above. Observe first that
    \[
    ((G_\ell)^{s\rtimes\sigma})^\circ = Z_G(s)^\circ
    \]
    since $1 \rtimes \sigma$ and $s \rtimes 1$ are both powers of $s \rtimes \sigma$. If $s \rtimes\sigma \in {}^h(L(\F_{q^\ell}) \rtimes \langle\sigma\rangle)$ for some $h \in G(\F_{q^\ell}) \rtimes \langle\sigma\rangle$, say $h = h_0 \rtimes \sigma^i$, then in particular (again using that $1 \rtimes \sigma$ is a power of $s \rtimes \sigma$) there is some $x \in L(\F_{q^\ell})$ such that $h_0x\sigma(h_0)^{-1} = 1$, so $h_0^{-1}\sigma(h_0) \in L(\F_{q^\ell})$. By Lang's theorem, there is some $x_0 \in L(\ol\F_q)$ such that $h_0^{-1}\sigma(h_0) = x_0^{-1}\Fr_q(x_0)$ (where $\Fr_q$ is the Frobenius endomorphism of $G$), and thus $h_0x_0^{-1} \in G(\F_q)$. It follows that $x_0 \in L(\F_{q^\ell})$ and that $^hL_{\F_{q^\ell}}$ is the base change to $\F_{q^\ell}$ of the twisted Levi $\F_q$-subgroup $(h_0x_0^{-1})L (h_0x_0^{-1})^{-1}$ of $G$. These considerations show that (by setting $g_0 = h_0x_0^{-1}$ and $t_0 = x_0^{-1}$) \eqref{eqn:char-formula-digne-michel} is equal to
    \begin{equation}\label{eqn:char-formula-massaged}
    \frac{1}{|Z_G(s)^\circ(\F_q)| \cdot |L(\F_{q^\ell})|} \sum_{\substack{g_0 \in G(\F_q) \\ t_0 \in L(\F_{q^\ell}) \\ s \in {}^{g_0}L(\F_q)}} \sum_{v \in Z_{^{g_0}L}(s)(\F_q)_{\mathrm{u}}} Q_{^{g_0}L}^{Z_G(s)^\circ}(u, v^{-1}) \cdot \BC'_\ell(\rho)(t_0g_0^{-1}s g_0\sigma(t_0^{-1}) \rtimes \sigma).
    \end{equation}
    Observe now that
    \[
    \BC'_\ell(\rho)(t_0g_0^{-1}s g_0\sigma(t_0^{-1}) \rtimes \sigma) = \rho(g_0^{-1}sg_0)
    \]
    since $\BC_\ell(\rho)$ is $\sigma$-stable and $\rho = \BC_\ell(\rho)|_{L(\F_q) \rtimes \{\sigma\}}$. Thus \eqref{eqn:char-formula-massaged} is equal to
    \[
    \frac{1}{|Z_G(s)^\circ(\F_q)|} \sum_{\substack{g_0 \in G(\F_q) \\ s \in {}^{g_0}L(\F_q)}} \sum_{v \in Z_{^{g_0}L}(s)(\F_q)_{\mathrm{u}}} Q_{^{g_0}L}^{Z_G(s)^\circ}(u, v^{-1}) \cdot \rho(g_0sg_0^{-1}).
    \]
    By \cite[Proposition 2.6]{DM94} again, this is equal to $R_L^G(\rho)(g)$.
\end{proof}

\begin{lemma}\label{lemma:psi-tilde}
    Let $\rho$ be an irreducible representation of $L(\F_q)$. Then the class function $R_{L_{\F_{q^\ell}}}^{G_{\F_{q^\ell}}}(\BC_\ell(\rho))$ on $G(\F_{q^\ell})$ admits an extension $\widetilde\psi$ to a class function on $G(\F_{q^\ell}) \rtimes \langle\sigma\rangle$ such that
    \[
    \widetilde\psi(g \rtimes \sigma) = R_L^G(\rho)(g)
    \]
    for all $g \in G(\F_q)$.
\end{lemma}

\begin{proof}
    Let $P_2 \subset (G_\ell)_{\ol\F_q}$ be the parabolic $\ol\F_q$-subgroup with Levi factor $(L_\ell)_{\ol\F_q}$ and unipotent radical $U \times {}^{\Fr_q}U \times \cdots \times {}^{\Fr_q^{\ell-1}}U$. By \cite[Corollaire 3.2]{Dig99}, we have
    \[
    R_{L_\ell, P_2}^{G_\ell}(\rho) = R_{L_{\F_{q^\ell}}}^{G_{\F_{q^\ell}}}(\rho).
    \]
    By \cite[Theorem 8.7.2]{GRV26}, we have $R_{L_\ell, P_2}^{G_\ell} = R_{L_\ell, P_1}^{G_\ell}$. By \cite[Corollaire 2.4]{DM94}, the image of $R_{L_\ell\rtimes\langle\sigma\rangle, P_2\rtimes\langle\sigma\rangle}^{G_\ell\rtimes\langle\sigma\rangle}(\rho)$ in $K_0(\Rep(G(\F_{q^\ell})))$ is equal to $R_{L_\ell,P_2}^{G_\ell}(\rho)$, so the result follows from Lemma~\ref{lemma:dig99-3.5}.
\end{proof}

\begin{remark}
    The proof of Lemma~\ref{lemma:psi-tilde} is the only point at which we use the difficult \cite[Theorem 8.7.2]{GRV26}, and here we use it only in a very special situation (following \cite{Dig99}). However, it is not clear whether this situation is any easier than the general case.
\end{remark}

We record the following lemma here for want of a more natural place to put it.

\begin{lemma}\label{lemma:tori-descend}
    Suppose $\ell > \rk G + 1$.
    \begin{enumerate}
        \item If $T$ is a maximal $\F_{q^\ell}$-torus of $G_{\F_{q^\ell}}$, then there exists a maximal $\F_q$-torus $T_0$ of $G$ such that $(T_0)_{\F_{q^\ell}}$ is $G(\F_{q^\ell})$-conjugate to $T$.
        \item If $T_0$ and $T_1$ are maximal $\F_q$-tori of $G$ and $\Transp_G(T_0, T_1)$ denotes the subscheme of $G$ parameterizing elements $g$ with $gT_0g^{-1} = T_1$, then
        \[
        \Transp_G(T_0, T_1)(\F_q)/T_1(\F_q) = \Transp_G(T_0, T_1)(\F_{q^\ell})/T_1(\F_{q^\ell}).
        \]
    \end{enumerate}
\end{lemma}

\begin{proof}
    If $S \subset G$ is a maximal $\F_q$-torus with Weyl group $W = (N_G(S)/S)(\ol\F_q)$, then the $G(\F_q)$-conjugacy classes of maximal $\F_q$-tori of $G$ correspond to $\Fr_q$-twisted conjugacy classes in $W$ in the following way: if $T_0$ is a maximal $\F_q$-torus of $G$, then by conjugacy of maximal $\ol\F_q$-tori of $G_{\ol\F_q}$, there is some $g \in G(\ol\F_q)$ such that $g(T_0)_{\ol\F_q}g^{-1} = S_{\ol\F_q}$. The element $g \cdot \Fr_q(g)^{-1}$ lies in $N_G(S)(\ol\F_q)$, and the $\Fr_q$-twisted conjugacy class of its image in $W$ is independent of the choice of $g$.
    
    Under the above identification, the base change map $T_0 \mapsto (T_0)_{\F_{q^\ell}}$ corresponds to the map $f\colon W \to W$ given by $f(w) = \prod_{i=0}^{\ell-1} \Fr_q^i(w)$, which intertwines $\Fr_q$-twisted conjugacy with $\Fr_q^\ell$-twisted conjugacy. Since $\ell > \dim S + 1$, the action of $\Fr_q$\footnote{In this context, $\Fr_q$ is the $q$-Frobenius automorphism of $\ol\F_q$, and in particular its action does \emph{not} agree with the action induced by the Frobenius endomorphism of $G$ (which we have also denoted $\Fr_q$ above).} on $X^*(S_{\ol\F_q})$ is of order prime to $\ell$, and thus the same is true of its action on $W$. Moreover, the classification of reductive groups and the assumption $\ell > \rk G + 1$ shows that $\ell$ does not divide the order of $W$, so $f$ is bijective, proving (1). Similar reasoning shows that $(N_G(S)/S)(\F_q) = (N_G(S)/S)(\F_{q^\ell})$.

    For (2), note that if $\Transp_G(T_0, T_1)(\F_q)$ is empty, then $\Transp_G(T_0, T_1)(\F_{q^\ell})$ is also empty, since the previous paragraph shows that base extension induces a bijection from the set of $G(\F_q)$-conjugacy classes of maximal $\F_q$-tori in $G$ to the set of $G(\F_{q^\ell})$-conjugacy classes of maximal $\F_{q^\ell}$-tori in $G_{\F_{q^\ell}}$. Otherwise, we have
    \[
    \Transp_G(T_0, T_1)(\F_q)/T_1(\F_q) = g \cdot N_G(T_1)(\F_q)/T_1(\F_q)
    \]
    for some $g \in \Transp_G(T_0, T_1)(\F_q)$, and the last sentence of the previous paragraph yields the claim.
\end{proof}

\section{Lusztig series}

We recall here the notion of \emph{geometric Lusztig series}. Let $\ell$ be a prime number not dividing $|G(\F_q)|$. If $\chi$ is an irreducible representation of $G(\F_q)$, then by \cite[Corollary 7.7]{DL76} there is a maximal $\F_q$-torus $T \subset G$ and a character $\theta\colon T(\F_q) \to \ol\Q_\ell^\times$ such that $\chi$ has nonzero pairing with $R_T^G(\theta)$. The \emph{geometric Lusztig series} $\cE(G, (T, \theta))$ corresponding to $(T, \theta)$ is the set of irreducible representations of $G(\F_q)$ which have nonzero pairing with $R_{T'}^G(\theta')$ for some pair $(T',\theta')$ which is geometrically conjugate to $(T,\theta)$ in the sense of \cite[Definition 5.5]{DL76}. By \cite[Corollary 6.3]{DL76}, the geometric Lusztig series form a partition of the set of irreducible representations of $G(\F_q)$.

\begin{lemma}\label{lemma:geometric-conjugacy-split}
    Let $(T, \theta)$ and $(T', \theta')$ be torus-character pairs in $G$. If $n$ is a positive integer such that $T_{\F_{q^n}}$ and $T'_{\F_{q^n}}$ are split, then $(T, \theta)$ is geometrically conjugate to $(T', \theta')$ if and only if $(T_{\F_{q^n}}, \theta \circ N_{\F_{q^n}/\F_q})$ is $G(\F_{q^n})$-conjugate to $(T'_{\F_{q^n}}, \theta' \circ N_{\F_{q^n}/\F_q})$.
\end{lemma}

\begin{proof}
    If $\Transp_G(T, T')$ denotes the scheme parameterizing local sections $g$ of $G$ conjugating $T$ to $T'$ and $W = N_G(T)/T$ is the Weyl group of $T$, then $\Transp_G(T, T')/T$ is a $W$-torsor. But $\Transp_G(T, T')(\F_{q^n})$ is nonempty by hypothesis, so $\Transp_G(T, T')/T$ is trivialized over $\F_{q^n}$. Since $W(\F_{q^n}) = W(\ol\F_q)$, the claim follows.
\end{proof}

In the remainder of this section, let $W_0$ be the absolute Weyl group of $G$.

\begin{lemma}\label{lemma:lusztig-series-size-bound}
    Let $T \subset G$ be a maximal $\F_q$-torus, and let $\theta\colon T(\F_q) \to \ol\Q_\ell^\times$ be a character. The geometric Lusztig series $\cE(G, (T, \theta))$ is of size at most $|W_0|^2$.
\end{lemma}

\begin{proof}
    Let $W = (N_{G}(T)/T)(\F_q)$ be the Weyl group of $(G, T)$. By \cite[Theorem 6.8]{DL76}, we have
    \[
    \langle R_{T}^{G}(\theta), R_{T}^{G}(\theta)\rangle_{G(\F_q)} = |\{w \in W\co {}^w\theta = \theta\}| \leq |W_0|.
    \]
    Note that the set of $G(\F_q)$-conjugacy classes of maximal $\F_q$-tori of $G$ is parameterized by the set of $\Gal(\ol\F_q/\F_q)$-twisted conjugacy classes in $W$, of which there at most $|W_0|$, so the lemma follows.
\end{proof}

\begin{lemma}\label{lemma:glauberman-lusztig-series-bijection}
    Let $T \subset G$ be a maximal $\F_q$-torus, let $\ell > |W_0|^2$ be a prime number not dividing $|G(\F_q)|$, let $\theta\colon T(\F_q) \to \ol\Q_\ell^\times$ be a character, and let $\theta_\ell = \BC_\ell(\theta)$. The Glauberman correspondence induces a bijection
    \[
    \cE(G, (T,\theta)) \xrightarrow[]{\cong} \cE(G_{\F_{q^\ell}}, (T_{\F_{q^\ell}},\theta_\ell)).
    \]
\end{lemma}

\begin{proof}
    By Lemma~\ref{lemma:dig99-3.5}, if $\chi$ is an irreducible constituent of $R_T^G(\theta)$ then the irreducible representation of $G(\F_{q^\ell})$ corresponding to $\chi$ under the Glauberman correspondence has nonzero pairing with $R_{T_{\F_{q^\ell}}}^{G_{\F_{q^\ell}}}(\theta_\ell)$. Thus the Glauberman correspondence sends $\cE(G, (T,\theta))$ to $\cE(G_{\F_{q^\ell}}, (T_{\F_{q^\ell}},\theta_\ell))$. Since the Glauberman correspondence is injective, it suffices to show that every element of $\cE(G_{\F_{q^\ell}}, (T_{\F_{q^\ell}}, \theta_\ell))$ is $\Gal(\F_{q^\ell}/\F_q)$-stable.
    
    Let $W = (N_{G}(T)/T)(\F_q)$ be the Weyl group of $(G, T)$. By \cite[Theorem 6.8]{DL76}, we have
    \[
    \langle R_{T}^{G}(\theta), R_{T}^{G}(\theta)\rangle_{G(\F_q)} = |\{w \in W\co {}^w\theta = \theta\}| \leq |W_0|.
    \]
    Since the pair $(T_{\F_{q^\ell}}, \theta_\ell)$ is $\Gal(\F_{q^\ell}/\F_q)$-stable, it follows that $\cE(G_{\F_{q^\ell}}, (T_{\F_{q^\ell}}, \theta_\ell))$ is $\Gal(\F_{q^\ell}/\F_q)$-stable. But the prime $\ell$ is larger than $|W_0|^2$ by assumption, so Lemma~\ref{lemma:lusztig-series-size-bound} shows that every element of $\cE(G_{\F_{q^\ell}}, (T_{\F_{q^\ell}}, \theta_\ell))$ is $\Gal(\F_{q^\ell}/\F_q)$-stable, as desired.
\end{proof}

\begin{lemma}\label{lemma:size-bound-for-torus-character-pairs}
    Let $(T, \theta)$ be a torus-character pair in $G$, and let $L \subset G$ be a twisted Levi $\F_q$-subgroup. The number of torus-character pairs $(S,\eta)$ in $L$ up to geometric conjugacy such that $(S, \eta)$ and $(T, \theta)$ are geometrically conjugate in $G$ is at most $|W_0|^2$.
\end{lemma}

\begin{proof}
    The number of maximal $\F_q$-tori $S$ in $L$ up to $L(\F_q)$-conjugacy is at most $|W_0|$. For each such $S$, there are at most $|W_0|$ characters $\eta$ of $S(\F_q)$ such that $(S, \eta)$ and $(T, \theta)$ are geometrically conjugate: indeed, this can be checked after passing from $(T, \theta)$ to $(T_{\F_{q^n}}, \theta \circ N_{\F_{q^n}/\F_q})$ for some $n$ such that $T_{\F_{q^n}}$ and $S_{\F_{q^n}}$ are both split, at which point it is clear. These observations prove the lemma.
\end{proof}

\begin{lemma}\label{lemma:mackey-formula-bound}
    There is a constant $C_G$, depending only on the root datum of $G$, such that for any twisted Levi $\F_q$-subgroup $L \subset G$ and any irreducible representation $\rho$ of $L(\F_q)$, we have
    \[
    |\langle R_L^G(\rho), R_L^G(\rho)\rangle_{G(\F_q)}| \leq C_G.
    \]
\end{lemma}

\begin{proof}
    We prove the result by induction of the semisimple rank of $G$, the case of tori being obvious (with $C_G = 1$). First, we show that if the result is true for a given $G$, then we may choose $C_G$ large enough such that if $q > C_G$ and $L \subset G$ is a twisted Levi $\F_q$-subgroup, then for every irreducible representation $\chi$ of $G(\F_q)$ we have
    \begin{equation}\label{eqn:dl-restriction-pairing-bound}
    |\langle {}^*R^G_L(\chi), {}^*R^G_L(\chi)\rangle_{L(\F_q)}| \leq C_G.
    \end{equation}
    By \cite[Corollaire 11.11]{Bon06}, if $\chi$ lies in the geometric Lusztig series $\cE(G, (T, \theta))$ then every irreducible constituent of ${}^*R^G_L(\chi)$ lies in a geometric series $\cE(L, (S,\eta))$, where $(S, \eta)$ is geometrically conjugate to $(T, \theta)$ as pairs in $G$. By Lemma~\ref{lemma:size-bound-for-torus-character-pairs}, the number of such pairs $(S, \eta)$ up to geometric conjugacy in $L$ is of order at most $|W_0|^2$. By Lemma~\ref{lemma:lusztig-series-size-bound}, each Lusztig series $\cE(L, (S,\eta))$ is of order at most $|W_0|^2$. Thus there is some $n \leq |W_0|^2$, irreducible characters $\rho_1, \dots, \rho_n$ of $L(\F_q)$, and $c_1, \dots, c_n \in \Z$, such that
    \[
    ^*R^G_L(\chi) = \sum_{i=1}^n c_i\rho_i.
    \]
    Observe that
    \begin{align*}
    |c_i| =|\langle {}^*R^G_L(\chi), \rho_i\rangle_{L(\F_q)}| &= |\langle \chi, R_L^G(\rho_i)\rangle_{G(\F_q)}| \\
        &\leq \langle R_L^G(\rho_i), R_L^G(\rho_i)\rangle_{G(\F_q)} \\
        &\leq C_G,
    \end{align*}
    so indeed
    \[
    \langle ^*R^G_L(\chi), {}^*R^G_L(\chi)\rangle_{L(\F_q)} = \sum_{i=1}^n c_i^2 \leq |W_0|^4 C_G^2.
    \]
    Replacing $C_G$ by $|W_0|^4 C_G^2$ (which still only depends on the root datum of $G$) allows us to assume \eqref{eqn:dl-restriction-pairing-bound} also holds when $G$ is replaced by any group of strictly smaller semisimple rank.
    
    By \cite[Th\'eor\`eme 5.1.1]{Bon98}, there are at most finitely many $q$ for which the Mackey formula does not hold for $G$ over $\F_q$. For each such $q$, there are only finitely many pairs $(L, \rho)$, and these only enforce a finite bound on $C_G$. Thus we may and do assume that the Mackey formula holds for $G$ over $\F_q$. Note that
    \begin{align*}
    \langle R_L^G(\rho), R_L^G(\rho)\rangle_{G(\F_q)} &= \langle {}^*R^G_L(R_L^G(\rho)), \rho\rangle_{L(\F_q)} \\
        &= \sum_{w \in N_G(L)(\F_q)/L(\F_q)} \langle R_{L \cap {}^wL}^L \circ {}^*R^{^wL}_{L\cap{}^wL} \circ \ad w(\rho), \rho\rangle_{L(\F_q)}. \addtocounter{equation}{1}\tag{\theequation} \label{eqn:dl-induction-self-pairing}
    \end{align*}
    Fix some $w \in N_G(L)(\F_q)/L(\F_q)$, and note that by the Cauchy--Schwarz inequality we have
    \begin{align*}
    |\langle R_{L \cap {}^wL}^L \circ {}^*R^{^wL}_{L\cap{}^wL} &\circ \ad w(\rho), \rho\rangle_{L(\F_q)}| \leq \\
    &\leq \langle R_{L \cap {}^w L} \circ {}^*R^{^wL}_{L\cap {}^wL} \circ \ad w(\rho), R_{L \cap {}^w L} \circ {}^*R^{^wL}_{L\cap {}^wL} \circ \ad w(\rho)\rangle_{L(\F_q)}. \addtocounter{equation}{1}\tag{\theequation} \label{eqn:cauchy-schwarz-pairing-1}
    \end{align*}
    Write ${}^*R^{^wL}_{L\cap {}^wL} \circ \ad w(\rho) = \sum_{i=1}^n c_i\eta_i$ for pairwise distinct irreducible representations $\eta_1, \dots, \eta_n$ of $(L \cap {}^wL)(\F_q)$ and $c_i \in \Z$, and note that
    \begin{equation}\label{eqn:dl-pairing-bound-2}
        \sum_{i=1}^n c_i^2 \leq C_L
    \end{equation}
    by \eqref{eqn:dl-restriction-pairing-bound}. Thus we have
    \begin{align*}
    \left\langle R_{L \cap {}^w L}\left(\sum_{i=1}^n c_i\eta_i\right), R_{L \cap {}^w L}\left(\sum_{i=1}^n c_i\eta_i\right)\right\rangle &= \sum_{1 \leq i, j \leq n} c_ic_j \langle R_{L \cap {}^wL}^L(\eta_i), R_{L\cap {}^wL}^L(\eta_j)\rangle \\
        &\leq \left(\sum_{i=1}^n c_i\sqrt{\langle R_{L\cap {}^wL}^L(\eta_i), R_{L\cap{}^wL}^L(\eta_i)\rangle} \right)^2 \\
        &\leq \left(\sum_{i=1}^n c_i\sqrt{C_L}\right)^2 \\
        &\leq C_L^3, \addtocounter{equation}{1}\tag{\theequation} \label{eqn:bound-on-mackey-self-pairing}
    \end{align*}
    where the first inequality follows from Cauchy--Schwarz, the second follows from the inductive hypothesis, and the third follows from \eqref{eqn:dl-pairing-bound-2}. Combining \eqref{eqn:dl-induction-self-pairing}, \eqref{eqn:cauchy-schwarz-pairing-1}, and \eqref{eqn:bound-on-mackey-self-pairing} shows that in the lemma statement we may take $C_G = \max_L(|W_0|C_L^3)$, where $L$ ranges over all proper twisted Levi $\F_q$-subgroups of $G$.
\end{proof}

\begin{remark}
    Granting the Mackey formula for $G$ over $\F_q$ (as we will prove), one can extract an explicit $C_G$ from the proof of Lemma~\ref{lemma:mackey-formula-bound}, namely
    \[
    C_G = |W_0|^{\frac{13}{5}(6^{r(G)}-1)},
    \]
    where $r(G)$ denotes the semisimple rank of $G$. It is clear even from the proof that this bound is very far from optimal. Assuming some conjectures in the theory of character sheaves, \cite[Proposition 13.2]{TT20} obtains a much better bound.
\end{remark}

Finally, we need the following elementary number-theoretic observation.

\begin{lemma}\label{lemma:elementary-number-theory}
    Let $K$ be a field of characteristic $0$ containing all $\ell$th roots of unity $1, \zeta, \zeta^2, \dots, \zeta^{\ell-1}$, and let $c_0, \dots, c_{\ell-1} \in \Z$ be such that
    \begin{enumerate}
        \item $\sum_{i=0}^{\ell-1} c_i\zeta^i \in \Z$,
        \item $\left|\sum_{i=0}^{\ell-1}c_i\right| + \left|\sum_{i=0}^{\ell-1} c_i\zeta^i\right| < \ell$.
    \end{enumerate}
    Then $c_1 = c_2 = \cdots = c_{\ell-1} = 0$.
\end{lemma}

\begin{proof}
    By (1), we have $c_1 = c_2 = \cdots = c_{\ell-1}$ and thus $\sum_{i=0}^{\ell-1} c_i\zeta^i = c_0 - c_1$ and $\sum_{i=0}^{\ell-1} c_i = c_0 + (\ell-1)c_1$.
    By the triangle inequality and (2), we have
    \[
    |\ell c_1| \leq |c_0 - c_1| + |c_0 + (\ell-1)c_1| < \ell,
    \]
    so $|c_1| < 1$ and thus $c_1 = 0$ since $c_1 \in \Z$.
\end{proof}

\section{Proofs of the main theorems}

In this section, we conclude the proofs of Theorems~\ref{theorem:main} and \ref{thm:base-change-commutes-with-lusztig-induction}. For convenience, we will say that a virtual representation $\rho \in K_0(\Rep(L(\F_q)))$ is \emph{irreducible} if either $\rho$ or $-\rho$ is the class of an irreducible representation of $L(\F_q)$.

\begin{proof}[Proof of Theorem~\ref{thm:base-change-commutes-with-lusztig-induction}]
    There are only finitely many twisted Levi $\F_q$-subgroups $L$ of $G$ up to $G(\F_q)$-conjugacy, so it is enough to fix one. Let $\rho$ be an irreducible representation of $L(\F_q)$, and suppose that $\rho$ lies in the geometric Lusztig series $\cE(L, (T,\theta))$. By \cite[Th\'eor\`eme 11.10]{Bon06}, every irreducible constituent of $R_L^G(\rho)$ lies in the geometric Lusztig series $\cE(G, (T, \theta))$. Write 
    \[
    R_L^G(\rho) = \sum_{i=1}^n a_i\chi_i,
    \]
    where $\cE(G, (T,\theta)) = \{\chi_1, \dots, \chi_n\}$. By Lemma~\ref{lemma:glauberman-lusztig-series-bijection}, if $\ell > |W_0|^2$ does not divide the order of $G(\F_q)$ then we may write
    \[
    R_{L_{\F_{q^\ell}}}^{G_{\F_{q^\ell}}}(\BC_\ell(\rho)) = \sum_{i=1}^n b_i\BC_\ell(\chi_i)
    \]
    for some $b_i \in \Z$. By Lemma~\ref{lemma:dig99-3.5} and \cite[Proposition 3.1]{Dig99}, there is a class function $\widetilde\psi$ on $G(\F_{q^\ell}) \rtimes \langle\sigma\rangle$ which extends $R_{L_{\F_{q^\ell}}}^{G_{\F_{q^\ell}}}(\BC_\ell(\rho))$ and satisfies
    \begin{equation}\label{eqn:psi-tilde-shear-restriction}
        \widetilde\psi|_{G(\F_q) \rtimes \{\sigma\}} = R_L^G(\rho).
    \end{equation}
    Let $\phi_0, \dots, \phi_{\ell-1}$ denote the characters $G(\F_{q^\ell}) \rtimes \langle\sigma\rangle \to \ol\Q_\ell^\times$ which vanish on $G(\F_{q^\ell}) \times \{1\}$, and suppose $\phi_0(1 \rtimes \sigma) = 1$. Each $\BC_\ell(\chi_i)$ extends to a character of $G(\F_{q^\ell}) \rtimes \langle\sigma\rangle$, and any two extensions differ by some $\phi_i$. Let $\widetilde\chi_i$ denote the unique extension of $\BC_\ell(\chi_i)$ to $G(\F_{q^\ell}) \rtimes \langle\sigma\rangle$ such that $\widetilde\chi_i(1 \rtimes \sigma) \in \Z$, and write
    \begin{equation}\label{eqn:psi-tilde-decomposition}
        \widetilde\psi = \sum_{i=1}^n \sum_{j=0}^{\ell-1} b_{ij} \widetilde{\chi}_i\phi_j.
    \end{equation}
    Since $\widetilde\psi$ extends $R_{L_{\F_{q^\ell}}}^{G_{\F_{q^\ell}}}(\BC_\ell(\rho))$ and \eqref{eqn:psi-tilde-decomposition} holds, we have
    \begin{enumerate}
        \item $\sum_{j=0}^{\ell-1} b_{ij} = b_i$ for all $i$, and
        \item $\sum_{j=0}^{\ell-1} b_{ij}\phi_j(\sigma) = a_i$ for all $i$.
    \end{enumerate}
    By \cite[Th\'eor\`eme 5.1.1]{Bon98}, there is a constant $D_G$ depending only on the root datum of $G$ such that for all $\ell > D_G$, the Mackey formula holds for $G$ over $\F_{q^\ell}$. By Lemma~\ref{lemma:mackey-formula-bound}, we may increase $D_G$ to assume that for all $\ell > D_G$ we have $|a_i| + |b_i| < \ell$ for all $i$. By Lemma~\ref{lemma:elementary-number-theory}, it follows that $b_{ij} = 0$ for all such $\ell$ and all $1 \leq j \leq \ell-1$. But then (1) and (2) above imply $a_i = b_i$ for all $i$, which proves \eqref{eqn:bc-commutes-with-dl-induction}.

    Now we aim to prove \eqref{eqn:bc-commutes-with-dl-restriction}. Let $\chi$ be an irreducible representation of $G(\F_q)$, and let $\rho_\ell$ be an irreducible representation of $L(\F_{q^\ell})$. It suffices to show
    \begin{equation}\label{eqn:bc-dl-restriction-pairing}
        \langle {}^*R^{G_{\F_{q^\ell}}}_{L_{\F_{q^\ell}}}(\BC_\ell(\chi)), \rho_\ell\rangle_{L(\F_{q^\ell})} = \langle \BC_\ell({}^*R^G_L(\chi)), \rho_\ell\rangle_{L(\F_{q^\ell})}.
    \end{equation}
    We have in any case
    \begin{equation}\label{eqn:bc-dl-adjunction}
    \langle {}^*R_{L_{\F_{q^\ell}}}^{G_{\F_{q^\ell}}}(\BC_\ell(\chi)), \rho_\ell\rangle_{L(\F_{q^\ell})} = \langle\BC_\ell(\chi), R_{L_{\F_{q^\ell}}}^{G_{\F_{q^\ell}}}(\rho_\ell)\rangle_{G(\F_{q^\ell})}
    \end{equation}
    by adjunction. There are now two possibilities:
    \begin{enumerate}
        \item The isomorphism class of $\rho_\ell$ is $\Gal(\F_{q^\ell}/\F_q)$-stable.
        \item The isomorphism class of $\rho_\ell$ is not $\Gal(\F_{q^\ell}/\F_q)$-stable.
    \end{enumerate}

    In case (1), \cite[Theorem 3]{Gla68} shows that there is a unique irreducible $\rho \in K_0(\Rep(L(\F_q)))$ such that $\rho_\ell = \BC_\ell(\rho)$. For $\ell > D_G$, we have then
    \begin{align*}
        \langle\BC_\ell(\chi), R_{L_{\F_{q^\ell}}}^{G_{\F_{q^\ell}}}(\rho_\ell)\rangle_{G(\F_{q^\ell})} &= \langle\BC_\ell(\chi), R_{L_{\F_{q^\ell}}}^{G_{\F_{q^\ell}}}(\BC_\ell(\rho))\rangle_{G(\F_{q^\ell})} \\
            &= \langle\BC_\ell(\chi), \BC_\ell(R_L^G(\rho))\rangle_{G(\F_{q^\ell})} \\
            &= \langle\chi, R_L^G(\rho)\rangle_{G(\F_q)} \\
            &= \langle {}^*R^G_L(\chi), \rho\rangle_{L(\F_q)} \\
            &= \langle\BC_\ell({}^*R^G_L(\chi)), \rho_\ell\rangle_{L(\F_{q^\ell})}
    \end{align*}
    since $\BC_\ell$ is an isometry. Combining this calculation with \eqref{eqn:bc-dl-adjunction} yields \eqref{eqn:bc-dl-restriction-pairing} in case (1).

    In case (2), the right hand side of \eqref{eqn:bc-dl-restriction-pairing} vanishes, so in view of \eqref{eqn:bc-dl-adjunction} the claim is that no irreducible constituent of $R_{L_{\F_{q^\ell}}}^{G_{\F_{q^\ell}}}(\rho_\ell)$ has $\Gal(\F_{q^\ell}/\F_q)$-stable isomorphism class. Suppose for the sake of contradiction that there is an irreducible constituent $\tau_\ell$ of $R_{L_{\F_{q^\ell}}}^{G_{\F_{q^\ell}}}(\rho_\ell)$ with $\Gal(\F_{q^\ell}/\F_q)$-stable isomorphism class. By \cite[Theorem 3]{Gla68}, there is an irreducible $\tau \in K_0(\Rep(G(\F_q)))$ such that $\tau_\ell = \BC_\ell(\tau)$. If $\pm\tau$ lies in the geometric Lusztig series $\cE(G, (T, \theta))$, then $\tau_\ell$ lies in $\cE(G_{\F_{q^\ell}}, (T_{\F_{q^\ell}}, \theta_\ell))$ by Lemma~\ref{lemma:glauberman-lusztig-series-bijection}, where $\theta_\ell = \BC_\ell(\theta)$. If $\rho_\ell$ lies in the Lusztig series $\cE(L_{\F_{q^\ell}}, (S, \eta))$, then by \cite[Th\'eor\`eme 11.10]{Bon06} the pair $(S, \eta)$ is geometrically conjugate to $(T_{\F_{q^\ell}}, \theta_\ell)$ in $G_{\F_{q^\ell}}$. Thus in view of Lemma~\ref{lemma:glauberman-lusztig-series-bijection} it suffices to show that there is a torus-character pair $(S_0, \eta_0)$ in $L$ such that $(S, \eta)$ is $L(\F_{q^\ell})$-conjugate to $((S_0)_{\F_{q^\ell}}, \BC_\ell(\eta_0))$.

    First observe that if $\ell > \rk G + 1$ then by Lemma~\ref{lemma:tori-descend}(1) there is a maximal $\F_q$-torus $S_0 \subset L$ such that $(S_0)_{\F_{q^\ell}}$ is $L(\F_{q^\ell})$-conjugate to $S$; we may therefore assume $S = (S_0)_{\F_{q^\ell}}$. In this case, the claim is that $\eta$ factors through $N_{\F_{q^\ell}/\F_q}$.
    
    Let $n$ be the least positive integer such that $S_{\F_{q^{\ell n}}}$ and $T_{\F_{q^{\ell n}}}$ are split, and note that every prime number dividing $n$ is at most $\dim S + 1$ (since the characteristic polynomial of $\Fr_q$ acting on $X^*(T_{\ol\F_q})$ is of degree $\dim T$); in particular, $\ell$ is prime to $n$. By Lemma~\ref{lemma:geometric-conjugacy-split}, the pairs $(S_{\F_{q^{\ell n}}}, \eta \circ N_{\F_{q^{\ell n}}/\F_{q^\ell}})$ and $(T_{\F_{q^{\ell n}}}, \theta_\ell \circ N_{\F_{q^{\ell n}}/\F_{q^\ell}})$ are $G(\F_{q^{\ell n}})$-conjugate. By Lemma~\ref{lemma:tori-descend}(2), there is some $g \in G(\F_{q^n})$ conjugating $(T_{\F_{q^{\ell n}}}, \theta_\ell \circ N_{\F_{q^{\ell n}}/\F_{q^\ell}})$ to $((S_0)_{\F_{q^{\ell n}}}, \eta \circ N_{\F_{q^{\ell n}}/\F_{q^\ell}})$. Since $\theta_\ell$ factors through $N_{\F_{q^\ell}/\F_q}$ and $\F_{q^n}$ is linearly disjoint from $\F_{q^\ell}$, it follows that $\eta \circ N_{\F_{q^{\ell n}}/\F_{q^\ell}}$ factors through $N_{\F_{q^{\ell n}}/\F_q}$, and thus $\eta$ factors through $N_{\F_{q^\ell}/\F_q}$, as desired.
\end{proof}

\begin{proof}[Proof of Theorem~\ref{theorem:main}]
    Let $\ell > \max(\rk G + 1, D_G)$ be a prime number, where $D_G$ is the constant from Theorem~\ref{thm:base-change-commutes-with-lusztig-induction}. It follows from Lemma~\ref{lemma:tori-descend} that
    \[
    L(\F_{q^\ell})\backslash S(L,M)(\F_{q^\ell})/M(\F_{q^\ell}) = L(\F_q)\backslash S(L,M)(\F_q)/M(\F_q).
    \]
    To prove \eqref{eqn:mackey-formula}, it is enough to prove it after composing with $\BC_\ell$. Since the Mackey formula is true for $G_{\F_{q^\ell}}$ by \cite[Th\'eor\`eme 5.1.1]{Bon98}, we have
    \begin{align*}
        \BC_\ell \circ {}^*R^G_L \circ R_M^G &= {}^*R^{G_{\F_{q^\ell}}}_{L_{\F_{q^\ell}}} \circ R_{M_{\F_{q^\ell}}}^{G_{\F_{q^\ell}}} \circ \BC_\ell \\
            &= \sum_{w \in L(\F_{q^\ell})\backslash S(L,M)(\F_{q^\ell})/M(\F_{q^\ell})} R_{(L\cap {}^wM)_{\F_{q^\ell}}}^{L_{\F_{q^\ell}}} \circ {}^*R^{^wM_{\F_{q^\ell}}}_{(L\cap {}^wM)_{\F_{q^\ell}}} \circ \ad w \circ \BC_\ell \\
            &= \sum_{w \in L(\F_q)\backslash S(L,M)(\F_q)/M(\F_q)} \BC_\ell \circ R_{L\cap{}^wM}^L \circ {}^*R^{^wM}_{L \cap {}^wM} \circ \ad w.
    \end{align*}
    This proves the claim.
\end{proof}

\bibliographystyle{halpha-abbrv}
\bibliography{bibliography}

\end{document}